\documentclass[12pt]{article}
\usepackage{amssymb}
\usepackage{mathrsfs}
\usepackage[hypertex]{hyperref}
\usepackage{amsfonts}
\usepackage{graphicx}
\usepackage{mathptmx}
\usepackage{latexsym,amsmath,amssymb,amsfonts,amsthm}
\newcommand{\R}{\mathbb{R}}

\newcommand{\bcen}{\begin{center}}
\newcommand{\ecen}{\end{center}}
\newtheorem{theorem}{Theorem}[section]
\newtheorem{lemma}[theorem]{Lemma}

\newtheorem{remark}[theorem]{Remark}

\def\Ric{{\rm Ric}}

\def\la{\Lambda}

\def\Li{\Lambda_i}
\def\Lj{\Lambda_j}
\def\Lk{\Lambda_{k+1}}

\def\p{\partial}
\def\n{\nabla}

\def\c{\cite}

\def\l{\label}
\def\a{\alpha}

\def\R{\mathbb{R}}

\def\si{\sum_{i=1}^{k}}
\def\sj{\sum_{j=1}^{k}}

\def\O{\Omega}
\def\La{\Lambda}

\def\div{{\mathrm{div}}}
\def\Ric{{\rm Ric}}

\def\R{\mathbb{R}}

\def\inm{\int_M}
\def\ino{\int_\Omega}
\def\a{\alpha}

\def\D{\Delta}
\def\L{\mathbb{L_{\phi}}}

\def\dm{d\mu}

\def\la{\langle}
\def\ra{\rangle}
\def\({\left(}
\def\){\right)}

\begin{document}
\setcounter{page}{1}
\title{Eigenvalue inequalities for the buckling problem of the
drifting Laplacian of arbitrary order}
\author{Feng Du$^{\dag}$,~ Lanbao Hou$^{\dag,\ddag}$,~Jing Mao$^{\ddag,\ast}$,~ Chuanxi Wu$^\ddag$}

\date{}
\protect\footnotetext{\!\!\!\!\!\!\!\!\!\!\!\!{$^{\ast}$Corresponding author}
{{\emph{E-mail addresses}}: defengdu123@163.com (F. Du),
houbao79@163.com (L.B. Hou), jiner120@163.com (J. Mao),
cxwu@hubu.edu.cn (C.X. Wu).}}
\maketitle ~~~\\[-15mm]

\begin{center}
{\footnotesize  $^{\dag}$School of Mathematics and Physics
Science,\\
Jingchu University of Technology, Jingmen, 448000, China\\
$^{\ddag}$Faculty of Mathematics and Statistics, \\
Key Laboratory of Applied Mathematics of Hubei Province, \\
Hubei University, Wuhan 430062, China }
\end{center}


\begin{abstract}
In this paper, we investigate the buckling problem of the drifting
Laplacian of arbitrary order on a bounded connected domain in
complete smooth metric measure spaces (SMMSs) supporting a special
function, and successfully get a general inequality for its
eigenvalues. By applying this general inequality, if the complete
SMMSs considered satisfy some curvature constraints, we can obtain a
universal inequalities for eigenvalues of this buckling problem.
 \end{abstract}

\markright{\sl\hfill F. Du, L.B. Hou, J. Mao, C.X. Wu  \hfill}

\section{Introduction}
\renewcommand{\thesection}{\arabic{section}}
\renewcommand{\theequation}{\thesection.\arabic{equation}}
\setcounter{equation}{0} \indent Let $\mathrm{\Omega}$ be a
 bounded domain in an $n$-dimensional complete Riemannian
manifold $M$, and let $\mathrm{\D}$ be the Laplace operator acting
on functions on $M$. Consider the following eigenvalue problems
\begin{eqnarray}\label{a1}
&& (-\D)^{m}u =-\Lambda \D u~~\mbox{in}
~~\Omega,~~~~u=\frac{\partial u}{\partial
\vec{\nu}}=\cdots=\frac{\partial^{m-1} u}{\partial
\vec{\nu}^{m-1}}=0~~\mbox{on}~~\partial
\Omega,\end{eqnarray}
 \begin{eqnarray}\label{a2} && (-\D)^{l}u
=\lambda  u~~\mbox{in} ~~\Omega,~~~~u=\frac{\partial u}{\partial
\vec{\nu}}=\cdots=\frac{\partial^{l-1} u}{\partial
\vec{\nu}^{l-1}}=0~~\mbox{on}~~\partial \Omega,
\end{eqnarray}
where  $\vec{\nu}$ is the outward unit normal vector field of the
boundary $\p \Omega$, $l$ is an arbitrary positive integer and $m$
is an arbitrary positive integer no less than 2. They are called
 \emph{the buckling problem of arbitrary order} and
\emph{the eigenvalue problem of polyharmonic operator},
respectively. The buckling problem (\ref{a1}) is used to describe
the critical buckling load of a clamped plate subjected to a uniform
compressive force around its boundary.

Denote by
\begin{eqnarray*}
&&0<\Lambda_1\leq\Lambda_2\leq\Lambda_3\leq\cdots,\\
&&0<\lambda_1\leq\lambda_2\leq\lambda_3\leq\cdots
\end{eqnarray*}
 the successive eigenvalues for (\ref{a1}) and (\ref{a2}) respectively, where each eigenvalue is repeated according to its multiplicity.
 An important theme in Geometric Analysis is to estimate these (and other) eigenvalues.

If $m=2$ in the bucking problem (\ref{a1}) and $\Omega$ is a bounded
domain in an $n$-dimensional Eucliden space $\mathbb{R}^n$,  Cheng
and Yang \c{CY1} proved the following universal inequality
\begin{eqnarray}\l{a3}
\si
(\La_{k+1}-\La_i)^2\leq\frac{4(n+2)}{n^2}\si(\La_{k+1}-\La_i)\La_{i},
\end{eqnarray}
which gives an answer to a long standing question proposed by Payne,
P\'olya and Weinberger \c{PPW1,PPW}.

If $m=2$ in (\ref{a1}) and $\Omega$ is a bounded domain in an
$n$-dimensional unit sphere $\mathbb{S}^n(1)$,  Wang and Xia \c{WX1}
successfully obtained the following universal inequality
\begin{eqnarray}\l{a4}
\nonumber2\si (\La_{k+1}-\La_i)^2&\leq&\si
(\La_{k+1}-\La_i)^2\(\frac{\delta^2\(\La_i-(n-2)\)}{4\(\Li+n-2\)}+\delta
\Li\)\\&&+\frac{1}{\delta}\si(\La_{k+1}-\La_i)\(\La_i+\frac{(n-2)^2}{4}\),
\end{eqnarray}
where $\delta$ is an arbitrary positive constant.

 Later, Cheng and Yang \c{CY2} gave an
 improvement for the universal inequalities (\ref{a3}) and (\ref{a4}) as
follows
\begin{eqnarray}\l{a5}
\si
(\La_{k+1}-\La_i)^2\leq\frac{4\(n+\frac{4}{3}\)}{n^2}\si(\La_{k+1}-\La_i)\La_i
\end{eqnarray}
and
\begin{eqnarray}\l{a6}
\nonumber&2&\si (\La_{k+1}-\La_i)^2+(n-2)\si
\frac{\(\La_{k+1}-\La_i\)^2}{\La_i-(n-2)}\\&\leq&\si
(\La_{k+1}-\La_i)^2\(\La_i-\(\frac{n-2}{\Li-(n-2)}\)
\)\delta_i+\si\frac{(\La_{k+1}-\La_i)}{\delta_i}\(\La_i+\frac{(n-2)^2}{4}\),
\qquad
\end{eqnarray}
where $\{\delta_i\}_{i=1}^{k}$ is an arbitrary positive
non-increasing monotone sequence.

For arbitrary $m$, when $\O$ is a bounded domain in a Euclidean
space or a unit sphere,  Jost, Li-Jost, Wang and Xia \cite{JLWX}
obtained some universal inequalities for eigenvalues of the buckling
problem (\ref{a1}), which have been improved by Cheng, Qi, Wang and
Xia \cite{CQWX} already.  For bounded domains of some special Ricci
flat manifolds considered in \cite{DWLX1} and of product manifolds
$\mathbb{M}\times \mathbb{R}$  considered in \cite{WX2} (with
$\mathbb{M}$ a complete Riemannian manifold), universal inequalities
for eigenvalues of the  buckling problem (\ref{a1}) have been
obtained therein. For some recent developments about universal
inequalities for eigenvalues of the eigenvalue problem (\ref{a2}) on
Riemannian manifolds, we refer to \cite{CIM1, CIM2,DWLX1,JLWX1} and
the references therein.

A smooth metric measure space (also known as the weighted measure
space, and here written as SMMS for short) is actually a Riemannian
manifold equipped with some measure which is conformal to the usual
Riemannian measure. More precisely, for a given complete
$n$-dimensional Riemannian manifold $(M,\langle,\rangle)$ with the
metric $\langle,\rangle$, the triple
$(M,\langle,\rangle,e^{-\phi}dv)$ is called a SMMS, where $\phi$ is
a \emph{smooth real-valued} function on $M$ and $dv$ is the
Riemannian volume element related to $\langle,\rangle$ (sometimes,
we also call $dv$ the volume density). On a SMMS
$(M,\langle,\rangle,e^{-\phi}dv)$, we can define the the so-called
\emph{drifting Laplacian} (also called \emph{weighted Laplacian})
$\L$ as follows
\begin{eqnarray*}
\L:=\Delta-\langle\nabla{\phi},\nabla\cdot\rangle
\end{eqnarray*}
where $\nabla$ is the gradient operator on $M$, and, as before,
$\Delta$ is the Laplace operator. Some interesting results
concerning eigenvalues of the drifting Laplacian can be found, for
instance, in \cite{CMZ,CZ,FHL,FS,MD,ML1,ML2}. On the SMMS
$(M,\langle,\rangle,e^{-\phi}dv)$, we can also define the so-called
\emph{$\infty$-Bakry-\'{E}mery Ricci tensor} $\mathrm{Ric}^{\phi}$
given by
\begin{eqnarray*}
\mathrm{Ric}^{\phi}=\mathrm{Ric}+\mathrm{Hess}\phi,
\end{eqnarray*}
which is also called the \emph{weighted Ricci curvature}. Here
$\mathrm{Ric}$,  $\mathrm{Hess}$ are the Ricci tensor and the
Hessian operator on $M$, respectively.
 The equation
$\mathrm{Ric}^\phi=\kappa\la,\ra$ for some constant $\kappa$ is just
the gradient Ricci soliton equation, which plays an important role
in the study of Ricci flow. For $\kappa= 0, \kappa> 0, \mathrm{or
}~\kappa< 0$, the gradient Ricci soliton $(M, \langle,\rangle,
e^{-\phi}dv, \kappa)$ is called steady, shrinking, or expanding
respectively. We refer readers to \cite{C} for some recent
interesting results about Ricci solitons.

Let $\O$ be a bounded domain in a complete SMMS
$(M,\langle,\rangle,e^{-\phi}dv)$. Consider the following eigenvalue
problem of the drifting Laplacian
\begin{eqnarray}\l{a7}
\left\{\begin{array}{ccc} \(-\L\)^lu
=\lambda  u,&&~~\mbox{in} ~~\Omega, \\[2mm]
u=\frac{\partial u}{\partial \vec{\nu}}=\cdots=\frac{\partial^{l-1}
u}{\partial \vec{\nu}^{l-1}}=0,&&~~\mbox{on}~~\partial \Omega,
\end{array}\right.
\end{eqnarray}
where, as before, $\vec{\nu}$ is the outward unit normal vector
field of the boundary $\p \Omega$ and $l$ is an arbitrary positive
integer. We know that $\L^l$ is self-adjoint on the space of
functions
 \begin{eqnarray*}
 {\mathbb F}=\left\{ f\in C^{l+2}(\O)\cap C^{l+1}(\p \O):
 f\Bigg|_{\p \O}=\frac{\p f}{\p \vec{\nu}}\Bigg|_{\p \O}=\cdots=\frac{\p^{l-1} f}{\p \vec{\nu}^{l-1}}\Bigg|_{\p \O}=0\right\}
  \end{eqnarray*}
  with respect to the inner product
 \begin{eqnarray*}
 \widetilde{\langle\langle f, g\rangle\rangle} =\int_\Omega f g e^{-\phi}dv:=\int_\Omega f g d\mu,
 \end{eqnarray*}
 and so the eigenvalue problem (\ref{a7}) has a discrete spectrum  whose elements are called eigenvalues and can be listed
 increasingly as follows
$$0< \lambda_1\leq\lambda_2 \leq \cdots \leq \lambda_k \leq \cdots ,$$
where each eigenvalue is repeated with its multiplicity.

For the eigenvalue problem (\ref{a7}), when $l=1$, Xia-Xu \cite{XX}
investigated the eigenvalues of the Dirichlet  problem of the
drifting Laplacian on compact manifolds and got some universal
inequalities; when $l=2$, Du, Wu, Li and Xia \cite{DWLX2} obtained
some universal inequalities of Yang type for eigenvalues of the
bi-drifting Laplacian problem either on a compact Riemannian
manifold with boundary (possibly empty) immersed in a Euclidean
space, a unit sphere or a projective space, or on bounded domains of
complete manifolds supporting some special function; when $l$ is an
arbitrary integer no less than 2, Pereira, Adriano and Pina
\cite{PAP} gave some universal inequalities on bounded domains in a
Euclidean space or a unit sphere, while Du, Mao, Wang and Wu
\cite{DMWW1} successfully obtained some universal inequalities on
bounded domains in the Guassian and cylinder solitons.

In this paper, we will consider the following \emph{buckling problem
(of the drifting Laplacian) of arbitrary order}
\begin{eqnarray}\l{a8}
\left\{\begin{array}{ccc} \(-\L\)^mu
=-\Lambda  \L u,&&~~\mbox{in} ~~\Omega, \\[2mm]
u=\frac{\partial u}{\partial \vec{\nu}}=\cdots=\frac{\partial^{m-1}
u}{\partial \vec{\nu}^{m-1}}=0,&&~~\mbox{on}~~\partial \Omega,
\end{array}\right.
\end{eqnarray}
where, as before, $\vec{\nu}$ is the outward unit normal vector
field of the boundary $\p \Omega$ and $m$ is an arbitrary positive
integer no less than 2. The eigenvalue problem (\ref{a8}) has
discrete spectrum (see Section \ref{sec2} for the details), which
can be listed increasingly as follows
$$0<\La_1\leq\La_2\leq\cdots\leq\La_k\leq\cdots,$$
where each eigenvalue is repeated with its multiplicity.

For the eigenvalue problem (\ref{a8}), we can prove:
\begin{theorem}  \label{theorem 4.1}
Let $(M, \langle,\rangle)$ be a complete connected Riemannian
manifold having weighted Ricci curvature $\Ric^\phi \geq 0$ for some
$\phi\in C^2(M)$, which is bounded above uniformly on $M$, and
containing a line. Then we have
\begin{eqnarray}\label{d3}
\nonumber\si(\Lk-\Li)^2&\leq&2\left\{(2m^2-3m+3)\si(\Lk-\Li)^2
\Li^{\frac{m-2}{m-1}}\right\}^{\frac{1}{2}}\\&&\times\left\{\si
(\Lk-\Li)\Li^{\frac{1}{m-1}}\right\}^{\frac{1}{2}}.
\end{eqnarray}
\end{theorem}

\begin{remark}
\rm{ (1) By (\ref{d3}), it is easy to get
\begin{eqnarray*}
\si(\Lk-\Li)^{2}\leq4(2m^2-3m+3)\Lambda_{k}^{\frac{m-2}{m-1}}\si
(\Lk-\Li)\Li^{\frac{1}{m-1}}
\end{eqnarray*}
 under the assumptions of Theorem \ref{theorem 4.1}. \\
(2) Especially, when $m=2$, under the assumptions of Theorem
\ref{theorem 4.1}, one has
\begin{eqnarray*}
\si(\Lk-\Li)^{2}\leq20\si (\Lk-\Li)\Li
\end{eqnarray*}
by applying (\ref{d3}) directly. \\
 (3) When $m=2$, the buckling problem (\ref{a8}) degenerates into
 the one considered in \cite{DMWW2} where Du, Mao, Wang and Wu firstly obtained
 universal inequalities on bounded connected domains on the Gaussian
 shrinking soliton
 $\left(\mathbb{R}^{n},\langle,\rangle_{can},e^{-\frac{1}{4}|x|}dv,\frac{1}{2}\right)$,
 with $x\in\mathbb{R}^{n}$, and on the gradient Ricci soliton $\left(\Sigma\times\mathbb{R},\langle,\rangle,e^{-\frac{\kappa
 t^2}{2}}dv,\kappa\right)$, with $\Sigma$ an Einstein manifold of constant Ricci
 curvature $\kappa$, $x\in\Sigma$ and $t\in\mathbb{R}$. In this
 sense, our universal inequality (\ref{d3}) here can be seen as a continuation of those in \cite{DMWW2}.
 }
\end{remark}

\section{Preliminaries} \label{sec2}
\renewcommand{\thesection}{\arabic{section}}
\renewcommand{\theequation}{\thesection.\arabic{equation}}
\setcounter{equation}{0} \indent In this section, firstly, inspired
by Cheng and Yang \c{CY1}, let us construct trial functions for the
buckling problem (\ref{a8}).

Let $\O$ be a bounded domain with smooth boundary in the complete
SMMS $(M,\widetilde{g},e^{-\phi}dv)$. Since any complete Riemannian
manifold can be isometrically embedded in some Euclidaen space, we
can treat our $M$ as a submanifold of some $\R^q$. Let us denote by
$\langle,\rangle$ the canonical metric on  $\R^q$ as well as that
induced on $M$. As before, $d\mu=e^{-\phi}dv$, denote by $\D$ and
$\n$ the Laplacian and the gradient operator of $M$, respectively.
Let $u_i$ be the $i$-th orthonormal eigenfunctions of the buckling
problem (\ref{a8}) corresponding to the eigenvalue $\Li$, namely,
$u_i$ satisfies
\begin{eqnarray}\label{b1}
\left\{\begin{array}{ccc} (-\L)^m u_i
=-\Lambda_i \L u_i,&&~~\mbox{in} ~~\Omega, \\[2mm]
u_i=\frac{\partial u_i}{\partial \nu}=\cdots=\frac{\partial u_i}{\partial \nu^{m-1}}=0,&&~~\mbox{on}~~\partial
\Omega,\\[2mm]
\ino \langle\n u_i,\n u_j\rangle\dm=\delta_{ij}.
\end{array}\right.
\end{eqnarray}

 For functions $f$ and $g$ on $\O$, the Dirichlet inner product $(f, g)_D$ of $f$ and $g$ is given by
$$(f, g)_D =\ino \langle\n f, \n g\rangle\dm.$$
The Dirichlet norm of a function $f$ is defined by
$$\|f\|_D=\{(f, f)_D\}^{\frac{1}{2}}=\(\ino |\n f|^2\dm\)^{\frac{1}{2}}.$$
Let $\n^k$ be the denote the $k$-th covariant derivative operator on
$M$, defined in the usual weak sense. For a function $f$ on $\O$,
the squared norm of $\n^kf$ is defined as (cf. \cite{H})
\begin{eqnarray*}
\left|\n^k f\right|^2=\sum_{i_1,\cdots,i_k=1}^n\(\n^k f(e_{i_1},\cdots, e_{i_k})\)^2,
\end{eqnarray*}
where $e_1, \cdots , e_n$ are orthonormal vector fields locally defined on $\O$. Define the Sobolev space $H^2_m(\O)$ by
\begin{eqnarray*}
H^2_m(\O)=\left\{f: f, |\n f|,\cdots, |\n^m f| \in L^2(\O)\right\}.
\end{eqnarray*}
Then $H^2_m(\O)$ is a Hilbert space with respect to the inner product $\langle\langle,\rangle\rangle$:
\begin{eqnarray*}
\langle\langle f, g\rangle\rangle=\ino \(\sum_{k=0}^m\n^k f\cdot\n^k
g\)\dm,
\end{eqnarray*}
where
$$\n^k f\cdot\n^k g=\sum_{i_1,\cdots,i_k=1}^n \n^k f(e_{i_1},\cdots, e_{i_k})\n^k g(e_{i_1},\cdots, e_{i_k}).$$

Consider the subspace $H_{m,D}^2(\O)$ of $H_m^2(\O)$ defined by
\begin{eqnarray*}
H^2_{m,D}(\O)=\left\{f\in H^2_m(\O): f\Bigg|_{\p\O}=\frac{\p f}{\p
\vec{\nu}}\Bigg|_{\p\O}\cdots=\frac{\p^{m-1} f}{\p
\vec{\nu}^{m-1}}\Bigg|_{\p\O}=0 \right\}.
\end{eqnarray*}
The poly-drifting Laplacian operator $\L^m$ defines a self-adjoint
operator acting on $H^2_{m,D}(\O)$ with discrete eigenvalues $0
\leq\La_1 \leq\cdots \leq \La_k \leq \cdots$ for the buckling
problem (\ref{a8}) and the eigenfunctions $\{u_i\}^{+\infty} _{i=1}$
defined in (\ref{b1}) form a complete orthonormal basis for the
Hilbert space $H^2_{2,D}(\O)$. If $\psi\in H^2_{2,D}(\O)$  satisfies
$\(\psi, u_j\)_D=0, \forall j=1,\cdots,k,$ then the Rayleigh-Ritz
inequality tells us that
\begin{eqnarray}\label{b2}
\La_{k+1}\|\psi\|^2_D\leq\ino \psi (-\L)^m\psi\dm.
\end{eqnarray}
For vector-valued functions $F = (f_1, f_2, \cdots , f_m), G = (g_1, g_2, \cdots , g_m): \O
 \rightarrow \mathbb{R}^q$, we define an inner product $(F,G)$ by
$$(F,G)=\ino \langle F, G\rangle\dm=\ino \sum_{\a=1}^qf_\a g_\a\dm.$$
The norm of $F$ is given by
$$\|F\|=(F,F)^{\frac{1}{2}}=\left(\ino \sum_{\a=1}^qf_\a^2\dm\right)^{\frac{1}{2}}.$$
Let $\mathbf{H}^2_m(\O)$ be the Hilbert space of vector-valued functions given by
\begin{eqnarray*}
\mathbf{H}^2_{m}(\O)=\left\{F = (f_1, f_2, \cdots , f_m): \O
 \rightarrow \mathbb{R}^q; f_\a: f_a, |\n f_\a|\in L^2(\O)~~\mathrm{for}~\a=1,\cdots,n \right\}.
\end{eqnarray*}
with norm $\|\cdot\|_1$:
$$\|F\|_1=\(\|F\|^2+\ino \sum_{\a=1}^q|\n f_\a|^2\dm\)^{\frac{1}{2}}.$$
Observe that a vector field on $\O$ can be regarded as a
vector-valued function from $\O$ to $\mathbb{R}^q$. Let
$\mathbf{H}^2_{m,D}(\O)$
 be a subspace of $\mathbf{H}^2_m(\O)$ spanned by the vector-valued functions $\{\n u_i\}_{i=1}^\infty$
which form a complete orthonormal basis of $\mathbf{H}^2_{m,D}(\O)$. For any $f \in H^2_{m,D}(\O)$, we have $\n f\in
\mathbf{H}^2_{m,D}(\O)$ and for any $X \in \mathbf{H}^2_{m,D}(\O)$, there exists a function $f \in H^2_{m,D}(\O)$ such that $X = \n f$.
Consider the function $g :\O
 \rightarrow \mathbb{R}$, then the vector fields $g\n u_i$ can be
decomposed as
\begin{eqnarray}\label{b3}
g\n u_i=\n h_{ i}+\mathbf{W}_{ i},
\end{eqnarray}
where $h_{ i} \in H_{m,D}^2(\Omega)$, $\n h_{ i}$ is the
projection of $g\n u_i$ in $\mathbf{H}_{m,D}^2(\Omega)$ and
$\mathbf{W}_{i}\bot \mathbf{H}_{m,D}^2(\Omega)$. Thus, we have,
for any function $h \in C^1(\Omega)\cap L^2(\Omega)$,
\begin{eqnarray*}
\ino \langle\mathbf{W}_{ i}, \n h\rangle\dm=0.
\end{eqnarray*}
Hence, $\mathbf{W}_{i}$ satisfies
\begin{eqnarray}\label{b4}
\left\{\begin{array}{ccc}\mathbf{W}_{ i}|_{\p \Omega}
=0, \\[2mm]
\div \mathbf{W}_{ i}-\langle\mathbf{W}_{ i},\n\phi\rangle=0.
\end{array}\right.
\end{eqnarray}

  At the end of this section, we would like to mention two facts.
  First,
 a simple calculation gives the following
Bochner formula for the drifting Laplacian (see \cite{WW}): for any
$f\in C^3(\O)$,
\begin{eqnarray*}
\frac{1}{2}\L|\n f|^2=|\n^2f|^2+\langle\n f, \n(\L
f)\rangle+\mathrm{Ric}^\phi(\n f,\n f).\end{eqnarray*}
 Hence, on the SMMS $(M, g, e^{-\phi}dv)$,
for any functions $f, g\in C^3(\O)$, we have
\begin{eqnarray*}
\L\langle \n f,\n g\rangle=2\langle\n^2f,\n^2g\rangle+\langle\n
f,\n(\L g)\rangle+\langle\n g,\n(\L f)\rangle+2\Ric^\phi\langle\n f,\n g\rangle.
\end{eqnarray*}
Furthermore, if $g$ satisfies $\n^2 g=0,\L g=0,\Ric^\phi\langle Y,\n
g\rangle$=0, where $Y$ is any vector field on $M$, then we have
\begin{eqnarray}\label{b5}
\L\langle \n g,\n f\rangle=\langle\n g,\n(\L f)\rangle.
\end{eqnarray}
Second, using a similar calculation as that in the proof of
\cite[Lemma 2.1]{JLWX}, we can get the following fact:
\begin{lemma}\label{lemma2.1}
Let $\La_i$ be the $i$-th, $i=1, 2, \cdots$, eigenvalue of the
eigenvalue problem (\ref{a8}), and $u_i$ be the orthonormal
eigenfunction corresponding to $\Lambda_i$. Then
\begin{eqnarray*}\label{b6}
0\leq\int_{\Omega} u_i(-\L)^k u_i d\mu\leq
(\La_i)^{\frac{k-1}{m-1}}, \qquad k=1,2, \cdots, m-1.
\end{eqnarray*}
\end{lemma}

\section{A universal inequality of the buckling problem (of
the drifting Laplacian) of arbitrary order}
\renewcommand{\thesection}{\arabic{section}}
\renewcommand{\theequation}{\thesection.\arabic{equation}}
\setcounter{equation}{0} \indent In this section, first, we will
give a general inequality on a bounded domain in SMMSs supporting a
special function.

\begin{theorem} \label{theorem3.1}
Given  an $n$-dimensional complete SMMS $(M, \langle,\rangle,
e^{-\phi}dv)$ supporting a function $g$ such that $\n^2g=0$, $\L
g=0$, $|\n g|^2=1$ and $\Ric^\phi\langle Y,\n g\rangle$=0 for any
vector field $Y\in \mathscr{X}(M)$, where $\mathscr{X}(M)$ denotes
the set of smooth vector fields on $M$. Let $\O$ be a bounded
connected domain in $M$ and let $\Lambda_i$ be the $i$-th eigenvalue
of the buckling problem (\ref{a8}). Then we have
\begin{eqnarray}\label{c1}
\nonumber&&\sum_{i=1}^k\(\La_{k+1}-\La_i\)^2\\\nonumber&\leq&\delta\sum_{i=1}^k\(\La_{k+1}-\La_i\)^2\Bigg\{(-1)^m\ino(-m+1)\Big(u_i\L^{m-1} u_i\\\nonumber&&+(2m^2-4m+3)\langle\n g,\n
u_i\rangle\left\langle\n g,\n\(\L^{m-2} u_i\)\right\rangle\Big)\dm
\\\nonumber&&-2\Lambda_i^{\frac{m-2}{m-1}}\|\langle\n g,\n u_i\rangle\|^2+\Lambda_i^{\frac{m-3}{m-1}}\|\n\langle\n g,\n u_i\rangle\|^2\Bigg\}\\&&+\sum_{i=1}^k\frac{\(\La_{k+1}-\La_i\)}{\delta}\|\n\langle\n g,\n u_i\rangle\|^2.
\end{eqnarray} where $\|f\|^2=\int_{\Omega} f^2\dm$.
\end{theorem}

\begin{proof} Consider the function $\psi_{ i}: \Omega\mapsto \R$
given by
\begin{eqnarray}\label{c2}
\psi_{ i}= h_{ i}-\sj b_{ ij}u_j,
\end{eqnarray}
where $ b_{ ij}=\ino g \langle \n u_i,\n u_j\rangle d\mu=b_{ ji}$,
and $h_i$ is determined by (\ref{b3}). It is easy to check that
$\psi_{ i}$ satisfies
$$\psi_i\Bigg|_{\p\O}=\frac{\p \psi_i}{\p\vec{\nu}}\Bigg|_{\p\O}=\cdots=\frac{\p^{m-1}
\psi_i}{\p\vec{\nu}^{m-1}}\Bigg|_{\p\O}=0,~~\mathrm{and}~~ \ino
\langle\n \psi_{ i}, \n u_j\rangle d\mu=0$$ for any $j=1, \cdots,
k$. It therefore follows from the Rayleigh-Ritz inequality  that
\begin{eqnarray}\label{c3}
\La_{k+1}\ino|\n\psi_{i}|^2\dm\leq\ino\psi_{ i}(-\L)^m\psi_{ i}\dm
,\qquad \forall i=1,\cdots,k.
\end{eqnarray}
From (\ref{c2}), we have
$\L \psi_{ i}=\L h_{i}-\sj b_{ ij}\L u_j,$
and
\begin{eqnarray}\label{c4}
(-\L)^m \psi_{ i}=(-\L)^m h_i+\sj\La_j b_{ ij}\L u_j.
\end{eqnarray}
Observing that $\ino \psi_{ i}\L u_j\dm=-\ino\langle \n\psi_{ i}, \n
u_j\rangle\dm=0,$ and $$\inm\langle\mathbf{W}_{ i}, \n
h\rangle\dm=0, \qquad \forall h\in C^1(\Omega)\cap L^2(\Omega), $$
we have
\begin{eqnarray}\label{c5}
\nonumber\ino\psi_{ i}(-\L)^m \psi_{ i}\dm&=&\ino\psi_{
i}(-\L)^mh_i\dm\\\nonumber&=&\ino(-\L)^mh_i(h_i-\sj
b_{ij}u_j)\dm\\\nonumber&=&\ino h_i(-\L)^mh_i\dm-\sj b_{ij}\ino u_j(-\L)^m
h_i\dm
\\\nonumber&=&\ino h_i(-\L)^mh_i\dm-\sj b_{ij}\ino (-\L)^m u_jh_i\dm\\\nonumber&=&\ino h_i(-\L)^m
h_i\dm-\sj \Lj b_{ij}\ino \langle\n u_j,\n h_i\rangle\dm
 \\\nonumber&=& \ino h_i(-\L)^mh_i\dm-\sj \Lj b_{ij}\ino\langle\n
u_j,g\n u_i-\mathbf{W}_i \rangle\dm\\&=&\ino h_i(-\L)^mh_i\dm-\sj \La_j
 b_{ ij}^2,
\end{eqnarray}
and
\begin{eqnarray}\label{c6}
\nonumber&&\ino h_i(-\L)^mh_i\\\nonumber&=&\ino \langle \n h_i, \n((-\L)^{m-1}
h_i)\rangle\dm\\\nonumber&=&\ino \langle g\n u_i-\mathbf{W}_i,
\n((-\L)^{m-1} h_i) \rangle\dm\\\nonumber&=&\ino \langle g\n u_i,
\n((-\L)^{m-1} h_i) \rangle\dm\\\nonumber&=&-\ino(\langle\n g,\n u_i\rangle+g\D
u_i)(-\L)^{m-1} h_i\dm\\\nonumber&=&-\ino\((-\L)^{m-2}\(\langle\n g,\n
u_i\rangle+g\L u_i\)\)(-\L)
h_i\dm\\\nonumber&=&-\ino\left\langle\n\((-\L)^{m-2}\(\langle\n g,\n
u_i\rangle+g\L u_i\)\),\n
h_i\right\rangle\dm\\\nonumber&=& -\ino\left\langle\n\((-\L)^{m-2}\(\langle\n g,\n
u_i\rangle+g\L u_i\)\),g\n
u_i-\mathbf{W}_i\right\rangle\dm
\\\nonumber&=& -\ino\left\langle\n\((-\L)^{m-2}\(\langle\n g,\n
u_i\rangle+g\L u_i\)\),g\n
u_i\right\rangle\dm\\&=& \ino
\(-\L\)^{m-2}\(\langle\n g,\n
u_i\rangle+g\L u_i\)\(\langle\n g,\n
u_i\rangle+g\L u_i\)\dm.
\end{eqnarray}

It follows from $\L g=0$ and (\ref{b5}) that
\begin{eqnarray*}
\L\(\langle\n g,\n
u_i\rangle+g\L u_i\)=g\L^2 u_i+\langle\n g,\n\(3\L u_i\)\rangle,
\end{eqnarray*}
which implies
\begin{eqnarray}\label{c7}
\L^{m-2}\(\langle\n g,\n
u_i\rangle+g\L u_i\)=g\L^{m-1} u_i+(2m-3)\langle\n g,\n\(\L^{m-2} u_i\)\rangle
\end{eqnarray}
Combining (\ref{c4})-(\ref{c7}), we have
\begin{eqnarray}\label{c8}
\nonumber&&\ino\psi_{ i}(-\L)^m \psi_{ i}\dm\\\nonumber
&=&(-1)^m\ino(2m-3)\(g\L u_i\left\langle\n g,\n\(\L^{m-2} u_i\)\right\rangle+\langle\n g,\n
u_i\rangle\left\langle\n g,\n\(\L^{m-2} u_i\)\right\rangle\)\dm
\\&&+(-1)^m\ino \(g^2\L u_i\L^{m-1}u_i+g\L^{m-1}u_i\langle\n g,\n
u_i\rangle\)\dm-\sj \La_j
 b_{ ij}^2.
\end{eqnarray}
On one hand,
\begin{eqnarray}\label{c9}
\nonumber&&\ino\(gu_i\langle\n g,\n(\L^{m-1}u_i)\rangle\)\dm\\\nonumber&=&\ino\(gu_i \L^{m-1}\langle\n g,\n u_i\rangle\)\dm
\\\nonumber&=&\ino\(\L^{m-1}(gu_i) \langle\n g,\n u_i\rangle\)\dm\\&=&\ino\((g \L^{m-1}u_i+2(m-1)\langle\n g,\n(\L^{m-1}u_i)\rangle) \langle\n g,\n u_i\rangle\)\dm.
\end{eqnarray}
On the other hand,
\begin{eqnarray}\label{c10}
\ino\(gu_i\langle\n g,\n(\L^{m-1}u_i)\rangle\)\dm=-\ino\L^{m-1}u_i\(u_i+g\langle\n g,\n u_i\rangle\)\dm.
\end{eqnarray}
Then we infer from (\ref{c9}) and (\ref{c10}) that
\begin{eqnarray*}\label{c11}
\ino\(gu_i\langle\n g,\n(\L^{m-1}u_i)\rangle\)\dm=\ino\((m-1)\langle\n g,\n(\L^{m-2}u_i)\rangle \langle\n g,\n u_i\rangle-\frac{1}{2}u_i\L u_i\)\dm.
\end{eqnarray*}
Hence
\begin{eqnarray}\label{c12}
\nonumber&&\ino g\L^{m-1}u_i\langle\n g,\n
u_i\rangle\dm\\
\nonumber &=&-\ino\(u_i\L^{m-1}u_i+gu_i\langle\n
g,\n(\L^{m-1}u_i)\rangle\)\dm
\\
 &=&-\ino\((m-1)\langle\n g,\n(\L^{m-2}u_i)\rangle \langle\n g,\n u_i\rangle+\frac{1}{2}u_i\L
 u_i\)\dm.
\end{eqnarray}
By a direct computation, we have
\begin{eqnarray}\label{c13}
&& \nonumber\ino g\L u_i\langle\n g,\n(\L^{m-2}u_i)\rangle\dm =\ino
g\L u_i\L^{m-2}\langle\n g,\n u_i\rangle\dm \qquad \qquad \qquad \\
\qquad \qquad \qquad\nonumber   &=&\ino
\L^{m-2}(g\L u_i)\langle\n g,\n u_i\rangle\dm\\
  \nonumber&=&
\ino\langle\n g,\n u_i\rangle\(2(m-2)\langle\n
g,\n(\L^{m-2}u_i)\rangle+\L^{m-1}u_i\)\dm\\
  &=&
\ino\((m-3)\langle\n g,\n u_i\rangle\langle\n
g,\n(\L^{m-2}u_i)\rangle-\frac{1}{2}u_i\L^{m-1}u_i\)\dm
\end{eqnarray}
and
\begin{eqnarray}\label{c14}
&&\nonumber\ino g^2\L u_i\L^{m-1}u_i\dm\\\nonumber&=&\ino u_i\L (g^2\L^{m-1}u_i)\dm
\\\nonumber&=&\ino u_i\(2\L^{m-1}u_i+g^2\L^m u_i+4g\langle\n g,\n (\L^{m-1}u_i)\rangle\)\dm
\\\nonumber&=&\ino \(2u_i\L^{m-1}u_i+(-1)^{m-1}\Lambda_ig^2u_i\L u_i+4gu_i\langle\n g,\n (\L^{m-1}u_i)\rangle\)\dm
\\\nonumber&=&\ino \(2u_i\L^{m-1}u_i+(-1)^{m-1}\Lambda_i(g^2|\n u_i|^2+u_i^2)+4gu_i\langle\n g,\n (\L^{m-1}u_i)\rangle\)\dm\\&=&\ino \((-1)^{m-1}\Lambda_i(g^2|\n u_i|^2-u_i^2)+4(m-1)\langle\n g,\n u_i\rangle\langle\n g,\n
(\L^{m-2}u_i)\rangle\)\dm.\qquad
\end{eqnarray}
Substituting (\ref{c12})-(\ref{c14}) into (\ref{c8}), we have
\begin{eqnarray}\label{c15}
\nonumber&&\ino\psi_{ i}(-\L)^m \psi_{ i}\dm\\\nonumber
&=&(-1)^m\ino(-m+1)\(u_i\L^{m-1} u_i+(2m^2-4m+3)\langle\n g,\n
u_i\rangle\left\langle\n g,\n\(\L^{m-2} u_i\)\right\rangle\)\dm
\\&&+\Lambda_i\ino \(g^2|\n u_i|^2+u_i^2\)\dm-\sj \La_j
 b_{ ij}^{2}.
\end{eqnarray}
It is easy to see that
\begin{eqnarray*}\label{c16}
\|g\n u_i\|^2=\|\n h_{ i}\|^2+\|\mathbf{W}_{ i}\|^2, \qquad \|\n h_{
i}\|^2=\|\n \psi_{ i}\|^2+\sj  b_{ ij}^2.
\end{eqnarray*}
Substituting the above equalities into  (\ref{c15}) yields
\begin{eqnarray}\label{c17}
\nonumber&&\(\La_{k+1}-\La_i\)\|\n \psi_i\|^2\\\nonumber
&=&(-1)^m\ino(-m+1)\(u_i\L^{m-1} u_i+(2m^2-4m+3)\langle\n g,\n
u_i\rangle\left\langle\n g,\n\(\L^{m-2} u_i\)\right\rangle\)\dm
\\&&-\Lambda_i\(\|u_i\|^2-\|\mathbf{W}_{ i}\|^2\)-\sj (\La_i-\La_j)
 b_{ ij}^2.
\end{eqnarray}
Let $\mathbf{A}_i=\n(gu_i-h_i)$. Then $u_i\n
g=\mathbf{A}_i-\mathbf{W}_i$, and we have
\begin{eqnarray}\label{c18}
\|u_i\|^2=\|u_i\n g\|^2=\|\mathbf{W}_i\|^2+\|\mathbf{A}_i\|^{2}.
\end{eqnarray}
Since $\ino \langle\n \langle\n g,\n u_i\rangle,\mathbf{W}_i\rangle\dm=0$, we can get
\begin{eqnarray*}
\nonumber2\|\langle\n g,\n u_i\rangle\|^2&=&-2\ino\langle u_i\n g,\n\langle\n g,\n u_i\rangle\rangle\dm
=-2\ino\langle \mathbf{A}_i,\n\langle\n g,\n u_i\rangle\rangle\dm
\\&\leq&\Lambda_i^{\frac{1}{m-1}}\|\mathbf{A}_i\|^2+\Lambda_i^{-\frac{1}{m-1}}\|\n\langle\n g,\n u_i\rangle\|^2,
\end{eqnarray*}
which implies
\begin{eqnarray}\label{c19}
-\Lambda_i\|\mathbf{A}_i\|^2\leq-2\Lambda_i^{\frac{m-2}{m-1}}\|\langle\n
g,\n u_i\rangle\|^2+\Lambda_i^{\frac{m-3}{m-1}}\|\n\langle\n g,\n
u_i\rangle\|^{2}.
\end{eqnarray}
Putting (\ref{c18}) and (\ref{c19}) into (\ref{c17}), we have
\begin{eqnarray}\label{c20}
\nonumber&&\(\La_{k+1}-\La_i\)\|\n \psi_i\|^2\\\nonumber
&=&(-1)^m\ino(-m+1)\(u_i\L^{m-1} u_i+(2m^2-4m+3)\langle\n g,\n
u_i\rangle\left\langle\n g,\n\(\L^{m-2} u_i\)\right\rangle\)\dm
\\&&-2\Lambda_i^{\frac{m-2}{m-1}}\|\langle\n g,\n u_i\rangle\|^2+\Lambda_i^{\frac{m-3}{m-1}}\|\n\langle\n g,\n u_i\rangle\|^2-\sj (\La_i-\La_j)
 b_{ ij}^2.
\end{eqnarray}

Setting
$Z_{ i}=\n\langle\n g,\n
u_i\rangle$,  we can obtain
\begin{eqnarray}\label{c21}
\nonumber c_{ ij}&=&\inm \la Z_{ i},\n
u_j\ra\dm=\inm\langle\n\langle\n g,\n u_i\rangle,\n
u_j\rangle\dm\\\nonumber&=&-\inm u_j\L\langle\n g,\n
u_i\rangle\dm\\\nonumber&=&-\inm u_j\langle\n g,\n (\L
u_i)\rangle\dm\\\nonumber&=&\inm \L u_i\langle\n g,\n
u_j\rangle\dm\\\nonumber&=& -\inm\langle\n\langle\n g,\n
u_j\rangle,\n
u_i\rangle\dm\\\nonumber&=&-\inm\left\langle\n\langle\n g,\n
u_j\rangle,\n u_i\right\rangle\dm\\&=&-c_{ji}.
\end{eqnarray}
and
\begin{eqnarray}
\nonumber -2\inm\langle g\n u_i,Z_{
i}\rangle\dm=-2\inm\langle g\n u_i,\n\langle\n g,\n
u_i\rangle\rangle=1.
\end{eqnarray}
On the other hand,
\begin{eqnarray}\label{c22}
\nonumber 1&=&-2\inm\la g\n u_i,Z_{ i}\ra\dm\\\nonumber&=&-2\inm\la\n h_{
i}+\mathbf{W}_{ i},Z_{ i}\ra\dm\\\nonumber&=&-2\inm\la\n h_{
i},Z_{ i}\ra\dm-2\inm\left\langle\mathbf{W}_{ i},\n\langle\n g,\n
u_i\rangle\right\rangle\dm\\\nonumber&=&-2\inm\langle\n
h_{ i},Z_{ i}\rangle\dm\\\nonumber&=&-2\inm\left\la\n \psi_{ i}+\sum_{j=1}^kb_{
ij}\n u_j,Z_{ i}\right\ra\dm\\\nonumber&=&-2\inm\left\la\n \psi_{ i},Z_{
i}\right\ra\dm-2\sum_{j=1}^kb_{ ij}c_{ ij}\\&=&-2\inm\left\la\n \psi_{ i},Z_{
i}-\sum_{j=1}^kc_{ ij}\n u_j\right\ra\dm-2\sum_{j=1}^kb_{ ij}c_{
ij}.
\end{eqnarray}
Then by the Schwarz inequality and (\ref{c20})-(\ref{c22}), for any
positive constant $\delta$, we can get
\begin{eqnarray}\label{c23}
\nonumber&&\(\La_{k+1}-\La_i\)^2\(1+2\sum_{j=1}^kb_{
ij}c_{ ij}\)\\\nonumber&=&\(\La_{k+1}-\La_i\)^2\left\{\inm
-2\left\la\n \psi_{ i},Z_{ i}-\sum_{j=1}^kc_{ij}\n
u_j\right\ra\right\}\\\nonumber&\leq&\delta\(\La_{k+1}-\La_i\)^3\|\n
\psi_{ i}\|^2+\frac{\(\La_{k+1}-\La_i\)}{\delta}\left\|Z_{
i}-\sum_{j=1}^kc_{ ij}\n
u_j\right\|^2\\\nonumber&\leq&\delta\(\La_{k+1}-\La_i\)^2\Bigg\{(-1)^m\ino(-m+1)\Big(u_i\L^{m-1} u_i\\\nonumber&&+(2m^2-4m+3)\langle\n g,\n
u_i\rangle\left\langle\n g,\n\(\L^{m-2} u_i\)\right\rangle\Big)\dm
\\\nonumber&&-2\Lambda_i^{\frac{m-2}{m-1}}\|\langle\n g,\n u_i\rangle\|^2+\Lambda_i^{\frac{m-3}{m-1}}\|\n\langle\n g,\n u_i\rangle\|^2\Bigg\}+\sj(\Li-\Lj)b_{ij}^2\\&&+\frac{\(\La_{k+1}-\La_i\)}{\delta}\|Z_{
i}\|^2+\frac{\(\La_{k+1}-\La_i\)}{\delta}\sum_{j=1}^kc_{
ij}^2.
\end{eqnarray}
Since $b_{ ij}=b_{ ji}, c_{ ij}=-c_{ ji},$ summing over $i$ from 1  to $k$ in
(\ref{c23}), the inequality
\begin{eqnarray*}\label{c24}
\nonumber&&\sum_{i=1}^k\(\La_{k+1}-\La_i\)^2\\\nonumber&\leq&\delta\sum_{i=1}^k\(\La_{k+1}-\La_i\)^2\Bigg\{(-1)^m\ino(-m+1)\Big(u_i\L^{m-1} u_i\\\nonumber&&+(2m^2-4m+3)\langle\n g,\n
u_i\rangle\left\langle\n g,\n\(\L^{m-2} u_i\)\right\rangle\Big)\dm
\\\nonumber&&-2\Lambda_i^{\frac{m-2}{m-1}}\|\langle\n g,\n u_i\rangle\|^2+\Lambda_i^{\frac{m-3}{m-1}}\|\n\langle\n g,\n u_i\rangle\|^2\Bigg\}\\&&+\sum_{i=1}^k\frac{\(\La_{k+1}-\La_i\)}{\delta}\|\n\langle\n g,\n u_i\rangle\|^2.
\end{eqnarray*}
holds.
This completes the proof of Theorem \ref{theorem3.1}.
\end{proof}

\begin{lemma}\label{lemma3.2}
Under the assumptions of Theorem \ref{theorem3.1}, we have
\begin{eqnarray}\label{c25}
\|\n\langle\n g,\n u_i\rangle\|^2\leq\Li^{\frac{1}{m-1}}
\end{eqnarray}
and
\begin{eqnarray}\label{c26}
\ino\langle\n g, \n((-\L)^{m-2} u_i)\rangle\langle\n g,\n
u_i\rangle\dm\leq\Li^{\frac{m-2}{m-1}}.
\end{eqnarray}
\end{lemma}
\begin{proof}
By Lemma \ref{lemma2.1} and applying the Schwarz inequality, we have
\begin{eqnarray}
\nonumber\|\n\langle\n g,\n u_i\rangle\|^2&=& \ino\langle\n g,\n
u_i\rangle\L\langle\n g,\n u_i\rangle\dm\\\nonumber&=& \ino\langle\n
g,\n u_i\rangle\langle\n g,\n \L u_i\rangle\dm\\\nonumber&=& \ino\L
u_i\langle\n g,\n\langle\n g,\n u_i\rangle\rangle\dm\\\nonumber&\leq&
\left\{\ino(\L u_i)^2\dm\right\}^{\frac{1}{2}}\left\{\ino\langle\n
g,\n\langle\n g,\n
u_i\rangle\rangle^2\dm\right\}^{\frac{1}{2}}\\\nonumber&\leq&
\left\{\ino(\L u_i)^2\dm\right\}^{\frac{1}{2}}\left\{\ino|\n
g|^2|\n\langle\n g,\n
u_i\rangle|^2\dm\right\}^{\frac{1}{2}}\\\nonumber&=& \left\{\ino\L^2
u_iu_i\dm\right\}^{\frac{1}{2}}\left\{\ino|\n\langle\n g,\n
u_i\rangle|^2\dm\right\}^{\frac{1}{2}}\\\nonumber&=&
\Li^{\frac{1}{2(m-1)}}\left\{\|\n\langle\n g,\n
u_i\rangle\|^2\right\}^{\frac{1}{2}},
\end{eqnarray}
which implies
\begin{eqnarray}
\|\n\langle\n g,\n u_i\rangle\|^2\leq\Li^{\frac{1}{m-1}}.
\end{eqnarray}
When $m=2p$, $p\in\mathbb{Z}_{+}$ with $\mathbb{Z}_{+}$ the set of
all positive integers, we have
\begin{eqnarray}
\nonumber&&\ino \left\langle\n g, \n \left\langle\n g,\n
\((-\L)^{p-1}u_i\)\right\rangle\right\rangle^2\dm\\\nonumber&\leq&\ino|\n
g|^2\left|\n \left\langle\n g,\n
\((-\L)^{p-1}u_i\)\right\rangle\right|^2\dm\\\nonumber&=&\ino
\left\langle\n g,\n \((-\L)^{p}u_i\)\right\rangle\left\langle\n g,\n
\((-\L)^{p-1}u_i\)\right\rangle\dm\\\nonumber&=&\ino-(-\L)^{p}u_i
\left\langle\n g,\n \left\langle\n g,\n
\((-\L)^{p-1}u_i\)\right\rangle\right\rangle\dm
\\\nonumber&\leq&\left\{\ino
\((-\L)^{p}u_i\)^2\dm\right\}^{\frac{1}{2}}\left\{\ino\left\langle\n
g,\n\left\langle\n g,\n
\((-\L)^{p-1}u_i\)\right\rangle\right\rangle^2\dm\right\}^{\frac{1}{2}}
\\\nonumber&\leq&\left\{\ino u_i
(-\L)^{2p}u_i\dm\right\}^{\frac{1}{2}}\left\{\ino|\n
g|^2\left|\n\left\langle\n g,\n
\((-\L)^{p-1}u_i\)\right\rangle\right|^2\dm\right\}^{\frac{1}{2}}
\\\nonumber&=&\Li^{\frac{1}{2}}\left\{\ino\left\langle\n
g,\n\left\langle\n g,\n
\((-\L)^{p-1}u_i\)\right\rangle\right\rangle^2\dm\right\}^{\frac{1}{2}},
\end{eqnarray}
which implies
 \begin{eqnarray} \label{c28}
  \ino \left\langle\n g, \n
\left\langle\n g,\n
\((-\L)^{p-1}u_i\)\right\rangle\right\rangle^2\dm\leq\Li.
\end{eqnarray}
 Together with (\ref{c28}), we have
\begin{eqnarray}\label{c29}
\nonumber &&\ino\langle\n g, \n\((-\L)^{m-2} u_i\)\rangle\langle\n
g,\n u_i\rangle\dm\\\nonumber &=&\ino\langle\n g, \n\((-\L)^{p-1}
u_i\)\rangle\langle\n g,\n
((-\L)^{p-1}u_i)\rangle\dm\\\nonumber&=&\ino-(-\L)^{p-1}
u_i\left\langle\n g, \left\langle\n g,\n
\((-\L)^{p-1}u_i\)\right\rangle\right\rangle\dm\\\nonumber
&\leq&\left\{\ino \((-\L)^{p-1}
u_i\)^2\dm\right\}^{\frac{1}{2}}\left\{\ino \left\langle\n g, \n
\left\langle\n g,\n
\((-\L)^{p-1}u_i\)\right\rangle\right\rangle^2\dm\right\}^{\frac{1}{2}}\\\nonumber
&\leq&\Li^{\frac{m-3}{2(m-1)}}\Li^{\frac{1}{2}}
=\Li^{\frac{m-2}{m-1}}.
\end{eqnarray}
When $m=2p+1$, $p\in\mathbb{Z}_{+}$,
 \begin{eqnarray} \nonumber&&\ino \left\langle\n g, \n
\left\langle\n g,\n
\((-\L)^{p-1}u_i\)\right\rangle\right\rangle^2\dm\\\nonumber&\leq&\ino|\n
g|^2\left|\n \left\langle\n g,\n
\((-\L)^{p-1}u_i\)\right\rangle\right|^2\dm\\\nonumber&=&\ino
\left\langle\n g,\n \((-\L)^{p}u_i\)\right\rangle\left\langle\n g,\n
\((-\L)^{p-1}u_i\)\right\rangle\dm\\\nonumber&=&\ino-(-\L)^{p}u_i
\left\langle\n g,\n \left\langle\n g,\n
\((-\L)^{p-1}u_i\)\right\rangle\right\rangle\dm
\\\nonumber&\leq&\left\{\ino
\((-\L)^{p}u_i\)^2\dm\right\}^{\frac{1}{2}}\left\{\ino\left\langle\n
g,\n\left\langle\n g,\n
\((-\L)^{p-1}u_i\)\right\rangle\right\rangle^2\dm\right\}^{\frac{1}{2}}
\\\nonumber&\leq&\left\{\ino
u_i(-\L)^{2p}u_i\dm\right\}^{\frac{1}{2}}\left\{\ino|\n
g|^2\left|\n\left\langle\n g,\n
\((-\L)^{p-1}u_i\)\right\rangle\right|^2\dm\right\}^{\frac{1}{2}}
\\\nonumber&\leq&\Li^{\frac{m-2}{2(m-1)}}\left\{\ino\left\langle\n
g,\n\left\langle\n g,\n
\((-\L)^{p-1}u_i\)\right\rangle\right\rangle^2\dm\right\}^{\frac{1}{2}},
\end{eqnarray}
which implies
 \begin{eqnarray}\label{c30} \ino \left\langle\n g,
\n \left\langle\n g,\n
\((-\L)^{p-1}u_i\)\right\rangle\right\rangle^2&\leq&\Li^{\frac{m-2}{m-1}}.\end{eqnarray}
Together with (\ref{c30}), we have
\begin{eqnarray}\label{c31}
\nonumber &&\ino\left\langle\n g, \n\((-\L)^{p-2}
u_i\)\right\rangle\langle\n g,\n u_i\rangle\dm=\ino\left\langle\n g,
\n\((-\L)^{p} u_i\)\right\rangle\left\langle\n g,\n
((-\L)^{p-1}u_i)\right\rangle\dm\\\nonumber&=&\ino-(-\L)^{p}
u_i\left\langle\n g, \left\langle\n g,\n
\((-\L)^{p-1}u_i\)\right\rangle\right\rangle\dm\\\nonumber
&\leq&\left\{\ino \((-\L)^{p}
u_i\)^2\dm\right\}^{\frac{1}{2}}\left\{\ino \left\langle\n g, \n
\left\langle\n g,\n
\((-\L)^{p-1}u_i\)\right\rangle\right\rangle^2\dm\right\}^{\frac{1}{2}}\\\nonumber
&\leq&\Li^{\frac{m-2}{2(m-1)}}\left\{\ino |\n g|^2\left|\n
\left\langle\n g,\n
\((-\L)^{p}u_i\)\right\rangle\right|^2\dm\right\}^{\frac{1}{2}}
\\&\leq&\Li^{\frac{m-2}{m-1}}.
\end{eqnarray}
It follows from (\ref{c29}) and (\ref{c31}) that
\begin{eqnarray*}\label{c32}
\ino\langle\n g, \n((-\L)^{m-2} u_i)\rangle\langle\n g,\n
u_i\rangle\dm\leq\Li^{\frac{m-2}{m-1}}.
\end{eqnarray*}
The proof is finished.
\end{proof}

Using (\ref{c1}), (\ref{c25}) and (\ref{c26}), we can get:
\begin{lemma}\label{lemma3.3}
Under the assumptions of Theorem \ref{theorem3.1}, we have
\begin{eqnarray*}\label{c33}
\nonumber\si(\Lk-\Li)^2&\leq&2\left\{(2m^2-3m+3)\si(\Lk-\Li)^2
\Li^{\frac{m-2}{m-1}}\right\}^{\frac{1}{2}}\\&&\times\left\{\si
(\Lk-\Li)\Li^{\frac{1}{m-1}}\right\}^{\frac{1}{2}}.
\end{eqnarray*}
\end{lemma}

Now, we would like to give the proof of Theorem \ref{theorem 4.1}.
However, before that, we need the following fact:

\begin{itemize}

\item \textbf{FACT} (\cite[Theorem 1.1]{FLZ})
\emph{Let $(M, \langle,\rangle)$ be a complete connected Riemannian
manifold with weighted Ricci curvature $\Ric^\phi \geq 0$ for some
$\phi\in C^2(M)$ which is bounded above uniformly on $M$. Then it
splits isometrically as $N \times \R^l$, where $N$ is some complete
Riemannian manifold without lines and $\R^l$ is the $l$-Euclidean
space. Furthermore, the function $\phi$ is constant on each $\R^l$
in this splitting.}

\end{itemize}
By \textbf{FACT}, if  $\Ric^\phi \geq 0$ for some $\phi\in C^2(M)$
which is bounded above uniformly on $M$, we know that
$(M,\langle,\rangle,\dm)$ splits isometrically as $N \times \R^l$.
Let $\overline{x}=(t, x)$ be the standard coordinate functions of
$N\times\R^l$, where $t\in N$ and $ x=(x_1,\cdots,x_l)\in \R^l$.
Since $\phi$ is constant on each $\R^l$ in this splitting, for
$\a=1,\cdots,l$ and for any vector field $Y\in \mathscr{X}(M)$, we
have
\begin{eqnarray}\label{d1}
\L x_\a=\D x_\a-\left\langle \n \phi, \n x_\a
\right\rangle=0,\quad~~|\n x_\a|=1,\quad~~\Ric^\phi(\n x_\a, Y)=0.
 \end{eqnarray}

 \begin{remark}
\rm{ In fact, the above \textbf{FACT} can be strengthened to be the
following version:
\begin{itemize}

\item (\cite[Theorem 6.1]{WW}, see also \cite{LC1,LC2}) \emph{If $\Ric^\phi \geq 0$ for some bounded $\phi$ and $M$ contains a line, then $M=N^{n-1} \times \R$ and $\phi$ is constant along the
line.}

\end{itemize}
Hence, one can find a special function $g$ on $M$ satisfying
(\ref{d1}) also (if $\Ric^\phi \geq 0$ for some bounded $\phi$ and
$M$ contains a line), which implies that a general inequality in
Theorem \ref{theorem3.1} can be used in the setting. In fact, if
$\Ric^\phi \geq 0$ for some bounded $\phi$ and $M$ contains a line,
 we know that $(M,\langle,\rangle,\dm)$
splits isometrically as $N \times \R$. Let $\overline{x}=(t, x_1)$
be the standard coordinate functions of $N\times\R$, where $t\in N$
and $ x_1\in \R$. As (\ref{d1}), for any vector field $Y\in
\mathscr{X}(M)$, we have
\begin{eqnarray*}
\L x_1=\D x_1-\left\langle \n \phi, \n x_1
\right\rangle=0,\quad~~|\n x_1|=1,\quad~~\Ric^\phi(\n x_1, Y)=0.
 \end{eqnarray*}
}
\end{remark}

\begin{proof} [Proof of Theorem \ref{theorem 4.1}]
Clearly, (\ref{d1}) shows the existence of special function
mentioned in Theorem \ref{theorem3.1} if the constraints on the
weighted Ricci curvature $\mathrm{Ric}^{\phi}$ and the weighted
function $\phi$ in Theorem \ref{theorem 4.1} were satisfied. Hence,
by using Lemma \ref{lemma3.3} (since $M$ contains a line, $l$ cannot
be zero), the universal inequality (\ref{d3}) for the buckling
problem (\ref{a8}) follows directly.
\end{proof}

\section*{Acknowledgments}
\renewcommand{\thesection}{\arabic{section}}
\renewcommand{\theequation}{\thesection.\arabic{equation}}
\setcounter{equation}{0} \setcounter{maintheorem}{0}

This work is partially supported by the NSF of China (Grant Nos.
11801496 and 11926352), the Fok Ying-Tung Education Foundation
(China), Hubei Key Laboratory of Applied Mathematics (Hubei
University), the NSF of Hubei Provincial Department of Education
(Grant Nos. B2019211, B2016261), Research Team Project of Jingchu
University of Technology (Grant No. TD202006) and Research Project
of Jingchu University of Technology (Grant No. QDB201608).

\end{document}